\documentclass{article}
\usepackage[english]{babel}
\usepackage{url}
\usepackage{caption}
\usepackage{subcaption}
\usepackage[usenames,dvipsnames]{color}
\usepackage[utf8]{inputenc}
\usepackage{float}
\usepackage{amssymb}
\usepackage{amsthm}
\usepackage[makeroom]{cancel}
\usepackage{amsmath}
\usepackage{amsfonts}
\usepackage{graphicx}
\usepackage{hyperref}
\usepackage[useregional]{datetime2}
\usepackage{enumitem}
\usepackage{algorithm}
\usepackage{algorithmic}
\usepackage{bm}
\usepackage{cite}
\usepackage{geometry}

\newcommand{\N}{\mathcal{N}}

\newtheorem{theorem}{Theorem}[section]

\newtheorem{lemma}[theorem]{Lemma}
\theoremstyle{definition}
\newtheorem{definition}{Definition}[section]
\theoremstyle{remark}
\newtheorem{remark}{Remark}

\allowdisplaybreaks

\title{\LARGE \bf
  Discrete-Time High Order Tuner With A Time-Varying Learning Rate\thanks{This work is supported by the Boeing Strategic University Initiative.}
}

\author{Yingnan Cui\thanks{Y. Cui and A.M. Annaswamy are with the Department of Mechanical Engineering, Massachusetts Institute of Technology, Cambridge, MA, 02139.} and Anuradha M. Annaswamy
}

\date{}

\begin{document}
\allowdisplaybreaks

\maketitle
\thispagestyle{empty}
\pagestyle{empty}

\begin{abstract}

  We propose a new discrete-time online parameter estimation algorithm that combines two different aspects, one that adds momentum, and another that includes a time-varying learning rate. It is well known that recursive least squares based approaches that include a time-varying gain can lead to exponential convergence of parameter errors under persistent excitation, while momentum-based approaches have demonstrated a fast convergence of tracking error towards zero with constant regressors. The question is when combined, will the filter from the momentum method come in the way of exponential convergence. This paper proves that exponential convergence of parameter is still possible with persistent excitation. Simulation results demonstrated competitive properties of the proposed algorithm compared to the recursive least squares algorithm with forgetting.



\end{abstract}


\section{Introduction}
\label{sec:introduction}
An essential part of any adaptive control algorithm is reliable, fast online parameter estimation\cite{Narendra2005,Goodwin_1984}. Beyond the basic gradient descent method, a large amount of works have focused on proposing provably stable, more efficient algorithms for online parameter estimation in adaptive control\cite{Gaudio_2021a,Gaudio20AC,bruce20,bruce20a,goel2020}.

It is well known that the introduction of a time-varying learning rate leads to exponential learning of the parameters in the presence of persistent excitation. Both recursive least squares (RLS) and RLS with forgetting have been frequently adopted for parameter estimation\cite{johnstone1982}. This idea of adopting time-varying learning rate has also led to some major breakthroughs in the optimization community. AdaGrad, for example, adapts the learning rate to the adjustment of parameters, applying larger updates for infrequently adjusted parameters and smaller updates for frequently adjusted parameters\cite{Duchi_2011}. AdaDelta adopts an exponential decaying average of the past gradients to address AdaGrad's aggressive, monotonically decaying learning rate\cite{Zeiler_2012}.

Yet another recent set of results that leads to accelerated performance, such as fast reduction of a loss function, is through the addition of momentum. It is a well observed fact that gradient descent method often performs badly around saddle points and local optima\cite{du2017gradient}, and provides a convergence rate in $\mathcal{O}(1/k)$, where $k$ is the iteration number. In contrast, Nesterov's acceleration, which adopts the idea of momentum, is a method that helps accelerate gradient descent and can lead to a convergence rate of $\mathcal{O}(1/k^2)$ when the loss function is smooth \cite{Nesterov_2018}. In problems of parameter estimation, it has been shown more recently that momentum-based methods, also known as high-order tuners (HT), can lead to acceleration even with time-varying regressors if the loss is strongly convex\cite{Gaudio20AC}.

In real-time systems, it is of paramount importance to have both acceleration in performance, i.e. in a fast decrease of the loss function, and in learning, i.e. fast convergence of the parameter estimates to their true values. The question therefore is if HT can be combined with time-varying learning rates and lead to both accelerated performance and accelerated learning. Since HT includes an additional filter, it needs to be ensured that the filtering action does not compromise the property of fast learning in the presence of time-varying gains. In this paper, we show that is not the case and that persistent excitation guarantees exponential convergence of the parameter estimates to the true value.

The specific HT that we consider is that based on Heavy Ball method (HB) that is proposed by Polyak\cite{Polyak64}. We add a time-varying gain matrix in addition to the momentum term that is present in the HB method. Through careful adjustment of the time-varying gain matrix, we show that the gain matrix remains bounded, does not go to zero with persistent excitation, and that the parameter estimates converge to their true values exponentially. This is the central contribution of this paper. All results are in the context of a nonlinear ARMA model with unknown parameters that are to be estimated.

Section \ref{sec:problem-statement} states the problem we want to solve. Section \ref{sec:algorithm} presents the algorithm. We discuss stability properties of the algorithm in section \ref{sec:stability-analysis} and show numerical simulations in section \ref{sec:numer-sim}. Section \ref{sec:conclusion} summarizes the paper and discusses future works.

\section{Problem Statement}
\label{sec:problem-statement}
We consider a class of discrete-time nonlinear plant models of the form
\begin{align}
  \label{eq:10}
  y_k = -\sum_{i=1}^{n}a_{i}^*y_{k-i} + \sum_{j=1}^{m}b_{j}^*u_{k-j-d} + \sum_{\ell=1}^pc^*_{\ell}f_\ell(y_{k-1}, \ldots, y_{k-n}, u_{k-1-d}, \ldots, u_{k-m-d}),
\end{align}
where $a_{i}^*$, $b_{j}^*$ and $c_{\ell}^*$ are unknown parameters that are constant and need to be identified, and $d$ is a known time-delay. The function $f_\ell$ is an analytic function and is assumed to be such that the system in \eqref{eq:10} is bounded-input-bounded-output (BIBO) stable. Denote $z_{k-1} = [y_{k-1}, \ldots, y_{k-n}]^\top$ and $v_{k-d-1} = [u_{k-1-d}, \ldots, u_{k-m-d}]^\top$. We rewrite \eqref{eq:10} in the form of a linear regression
\begin{equation}
  \label{eq:11}
  y_k = \phi_k^\top\theta^*,
\end{equation}
where $\phi_k = [z_{k-1}^\top, v_{k-d-1}^\top, f_1(z_{k-1}^\top, v_{k-d-1}^\top), \ldots, \allowbreak f_p(z_{k-1}^\top, v_{k-d-1}^\top)]^\top$ is a regressor determined by exogenous signals and $\theta^* = [a_{1}^*, \ldots, a_{n}^*, b_{1}^*, \ldots, b_{m}^*,c_{1}^*, \ldots, c_{\ell}^*]^\top$ is the underlying unknown parameter vector. We propose to identify the parameter $\theta^*$ as $\theta_k$ using an estimator
\begin{equation}
  \hat{y}_k = \phi_k^\top\theta_k,
  \label{eq:estimate}
\end{equation}
which leads to a prediction error
\begin{equation}
  \label{eq:12}
  e_{y,k} = \phi_k^{\top}\tilde{\theta}_k,
\end{equation}
where $e_{y,k} = \hat{y}_k - y_k$ is the output prediction error and $\tilde{\theta}_k = \theta_k - \theta^*$ is the parameter error. The goal of parameter identification is to design an iterative procedure such that the parameter error $\|\tilde{\theta}_k\|$ converges to zero exponentially fast.

The iterative procedure for estimating the parameters is based on a squared loss function,
\begin{equation}
  \label{eq:7}
  L_k(\theta_k) = \frac{1}{2}e_{y,k}^2 = \frac{1}{2}\tilde{\theta}_k^\top\phi_k\phi_k^\top\tilde{\theta}_k,
\end{equation}
where the subscript $k$ in $L_k$ denotes $k$th iteration. In the literature, a normalized gradient descent algorithm has been shown to be stable although having a slow convergence rate \cite{Goodwin_1984}
\begin{equation}
  \label{eq:gd}
  \theta_{k+1} = \theta_k - \alpha\frac{\nabla L_k(\theta_k)}{\N_k}, \quad 0 < \alpha < 2,
\end{equation}
where $\N_k$ is a normalizing signal and is defined as $\N_k = 1 + \|\phi_k\|^2$.

The following definitions will be utilized for proving the main results.

\begin{definition}
  The regressor $\phi_k$ is said to satisfy the persistent excitation (PE) condition over an interval $\Delta T$, if for all $k \geq 0$,
  \begin{equation}
    \label{eq:pe}
    \epsilon_1 I \leq \sum_{i=k-\Delta T}^{k-1}\phi_i\phi_i^\top \leq \epsilon_2 I.
  \end{equation}
  \label{def:pe}
\end{definition}
\begin{definition}[From \cite{Luenberger1997}]
  For any fixed $p\in[1,\infty)$, a sequence of scalars $\xi=\{\xi_0, \xi_1, \ldots\}$ is defined to belong to $\ell_p$ if
  \begin{equation}
    \label{eq:102}
    \|\xi\|_\infty \equiv \left(\lim_{k\rightarrow\infty}\sum_{i=0}^k\|\xi_i\|^p\right)^{1/p} < \infty.
  \end{equation}
  When $p=\infty$, $\xi\in\ell_\infty$ if
  \begin{equation}
    \label{eq:103}
    \|\xi\|_{\ell_\infty} \equiv \sup_{i\geq 0}\|\xi_i\| < \infty
  \end{equation}
\end{definition}

Let $k \geq 0$ and consider the following time-varying dynamic system
\begin{equation}
  \label{eq:8}
  x_{k+1} = f(k, x_k),
\end{equation}
where $x_k\in\mathcal{D}$, $k\geq 0$, $\mathcal{D}$ is an open set such that $0 \in \mathcal{D}$, $f: \mathbb{N}\times \mathcal{D} \rightarrow\mathbb{R}^n$ is continuous and for all $k\in\mathbb{N}$, $f(k, 0) = 0$. The following definition and theorem of uniform global exponential stability is modified from \cite[Page 783-785]{chellaboina2008}.




\begin{definition}[Uniform global exponential stability]
  The origin in \eqref{eq:8} is uniformly globally exponentially stable if there exist scalars $c_1 > 0$ and $c_2 > 1$ such that $\|x_k\| \leq c_1 \|x_0\|\exp(-c_2^{-1}k)$, for all $x_0\in\mathbb{R}^n$.
  \label{def:uggs}
\end{definition}


\begin{theorem}
  If there exist a continuous function $V:\mathbb{N}^+ \times \cal{D} \rightarrow \mathbb{R}$ and positive constants $\bar\alpha, \bar\beta, \bar\gamma$ such that
  \begin{align}
    \label{eq:9}
    \bar\alpha\|x\|^2 \leq V(k, x) \leq \bar\beta \|x\|^2, \quad &(k, x) \in \mathbb{N}^+ \times \cal{D}, \\
    \label{eq:13}
    \Delta V \leq -\bar\gamma\|x\|^2, \quad &(k, x) \in \mathbb{N}^+ \times \cal{D},
  \end{align}
  then the origin in \eqref{eq:8} is uniformly globally exponentially stable.
  \label{theo:geo-stable}
\end{theorem}

\section{The Algorithm}
\label{sec:algorithm}
The Heavy Ball method, initially proposed in \cite{Polyak64}, achieves acceleration by adding a momentum term in addition to normalized gradient descent method
\begin{equation}
  \label{eq:1}
  \theta_{k+1} = \theta_k - \bar{\gamma}\frac{\nabla L_k(\theta_k)}{\N_k} + \bar\beta(\theta_k - \theta_{k-1}), 
\end{equation}
where $\bar\gamma$ is the learning rate constant and $\bar\beta$ is a constant that controls the momentum. In this work, we consider a time-varying matrix $F_k$ as an alternative to the constant $\bar\gamma$ in an effort to not only achieve fast convergence of the output error to zero but also have parameter error $\tilde{\theta}_k$ to zero. We propose the resulting algorithm as
\begin{align}
  \vartheta_{k+1} &= \vartheta_k - F_{k}\frac{\nabla L_k(\theta_{k+1})}{\N_k},\label{eq:2} \\
  \theta_{k+1} &= \theta_k - \beta(\theta_k - \vartheta_k),\label{eq:3}
\end{align}
where $\N_k = 1 + \eta\phi_k^\top F_{k-1}\phi_k$ and $F_k$ is updated as
\begin{equation}
  \label{eq:4}
  F_{k} = \lambda\left(F_{k-1} - \kappa\frac{F_{k-1}\phi_k\phi_k^\top F_{k-1}}{\N_k}\right).
\end{equation}
In \eqref{eq:2}, \eqref{eq:3} and \eqref{eq:4}, $\lambda$, $\kappa$, $\beta$ and $\eta \geq \kappa$ are positive hyperparameters whose bounds will be defined later. The update of $F_k$ is similar to the covariance matrix update in recursive least squares (RLS) algorithm with forgetting\cite{Goodwin_1984} but differs in the choice of the normalization and in the update of $F_k$. The main contribution of this paper is to show that the algorithm in \eqref{eq:2}-\eqref{eq:4} results in exponential convergence under PE.

\section{Stability Analysis}
\label{sec:stability-analysis}
In this section, we show that the algorithm in \eqref{eq:2}, \eqref{eq:3} and \eqref{eq:4} guarantees exponential convergence for suitable choices of the hyperparameters $\lambda$, $\kappa$, $\beta$ and $\eta$. Let
\begin{equation}
  \label{eq:mu}
  \mu = \min\{c_1, c_2\},
\end{equation}
where
\begin{align}
  \label{eq:c1}
  c_1 &= \left(1 - \frac{1}{\lambda}\right)F_{\max}^{-1} \geq 0,\\
  \label{eq:c2}
  c_2 &= \left\{1 - (1 - \beta)^2\left[\frac{1}{\lambda} + \frac{\kappa}{\lambda(\eta - \kappa)} + \frac{4\lambda}{\eta^2}\right]\right\}F_{\max}^{-1} \geq 0,
\end{align}
and $F_{\max}$ is the upper bound of $F_k$ under the persistent excitation in Definition \ref{def:pe}. When $\lambda = 1$, there is no forgetting in $F_k$ and from \eqref{eq:4}, $F_k\leq F_{k-1}$. Therefore $F_{\max} = \sigma_{\max}\{F_0\}$. When $\lambda > 1$, the following lemma gives the upper bound for $F_k$.
\begin{lemma}
  When the regressor $\phi_k$ satisfies PE condition in Definition \ref{def:pe}, the hyperparameters in \eqref{eq:1}, \eqref{eq:2} and \eqref{eq:3} satisfy $\frac{\kappa\epsilon_1(\lambda - 1)}{\lambda(\lambda^{\Delta T} - 1)} > (\eta - \kappa)\max_{i}\|\phi_i\|^2$ and $F_0 \leq \frac{F_{\max}}{\lambda^{\Delta T - 1}}I$, there exists $F_{\max}^{-1} = \frac{\kappa\epsilon_1(\lambda - 1)}{\lambda(\lambda^{\Delta T} - 1)} - (\eta - \kappa)\max_{i}\|\phi_i\|^2 \in \mathbb{R}^+$ such that $F_k \leq F_{\max}I$ for all $k \geq 0$.
  \label{lemma:1}
\end{lemma}

\begin{proof}
  Denote $\Delta_{\max} = 1 + (\eta - \kappa)F_{\max}\max_{i}\|\phi_i\|^2$. From \eqref{eq:4}, for all $k\geq 0$, $F_k^{-1} \leq \lambda F_{k+1}^{-1}$ and $\phi_k\phi_k^\top / \Delta_{\max} \leq \lambda F_k^{-1} / \kappa$. Since $F_0 \leq \frac{F_{\max}}{\lambda^{\Delta T - 1}}I$, $F_{k} \leq F_{\max}I$ for all $0\leq k \leq \Delta T - 1$. For all $k\geq \Delta T$, apply the PE definition, we obtain
  \begin{align*}
    \frac{\epsilon_1 I}{\Delta_{\max}} &\leq \sum_{i=k-\Delta T}^{k-1} \frac{\phi_i\phi_i^\top}{\Delta_{\max}} \\
                                       &\leq \frac{\lambda}{\kappa} (1 + \lambda + \cdots + \lambda^{\Delta T - 1}) F_{k-1}^{-1} \\
                                       &\leq \frac{\lambda}{\kappa} \frac{\lambda^{\Delta T} - 1}{\lambda - 1} F_{k-1}^{-1}
  \end{align*}
  Therefore $F_{k-1}^{-1} \geq F_{\max}^{-1}I$ for all $k \geq \Delta T$.
\end{proof}
\begin{remark}
  Due to the differences in the update of the learning rates between our algorithm and RLS, certain constraints on the hyperparameters have to be assumed for proof of the upper bound of $F_k$. This is mainly due to the choice of the denominator $\N_k$ in \eqref{eq:4}.
\end{remark}
When $\lambda = 1$, it can be shown that under PE $F_k\rightarrow 0$ as $k\rightarrow\infty$. The following lemma gives a lower bound on $F_k$ under PE when $\lambda > 1$.
\begin{lemma}
  When the regressor $\phi_k$ satisfies PE condition in Definition \ref{def:pe}, there exists $F_{\min} \in \mathbb{R}^+$, where $F_{\min}^{-1}I = F_{\Delta T - 1}^{-1} + \frac{\kappa\epsilon_2 I}{\lambda(1 - 1 / \lambda^{\Delta T})}$, such that $F_k \geq F_{\min}$ for all $k \geq 0$.
\end{lemma}

\begin{proof}
  From \eqref{eq:4}, $F_k^{-1} \leq \lambda F_{k+1}^{-1}$, therefore for all $k \geq \Delta T$,
  \begin{align*}
    F_k^{-1} &\leq \frac{1 - 1 / \lambda}{1 - 1 / \lambda^{\Delta T}} \sum_{i=k-1}^{k+\Delta T - 2} F_{i+1}^{-1} \\
    &\leq \frac{1 - 1 / \lambda}{1 - 1 / \lambda^{\Delta T}} \left(\frac{1}{\lambda}\sum_{i=k-1}^{k+\Delta T - 2} F_i^{-1} + \frac{\kappa}{\lambda}\epsilon_2 I\right) \\
    &\leq \frac{1 - 1 / \lambda}{1 - 1 / \lambda^{\Delta T}} \left(\frac{1}{\lambda^k}\sum_{i=0}^{\Delta T - 1} F_i^{-1} + \frac{\kappa}{\lambda}\frac{1 - 1 / \lambda^{k}}{1 - 1 / \lambda} \epsilon_2 I\right) \\
    &\leq \lambda^{\Delta T - k} F_{\Delta T - 1}^{-1} + \frac{\kappa}{\lambda}\frac{1 - 1 / \lambda^k}{1 - 1 / \lambda^{\Delta T}}\epsilon_2 I \\
    &\leq F_{\Delta T - 1}^{-1} + \frac{\kappa\epsilon_2 I}{\lambda(1 - 1 / \lambda^{\Delta T})} \\
    &= F_{\min}^{-1}I
  \end{align*}
\end{proof}

\begin{remark}
  When $\lambda > 1$, from the expressions of $F_{\max}$ and $F_{\min}$, we can observe that $F_{\max}$ and $\epsilon_1$ are inversely correlated, $F_{\min}$ and $\epsilon_2$ are inversely correlated. In the presence of weak excitation signals, $F_{\max}$ increases and can potentially become infinite, which is similar to the covariance matrix update in RLS with forgetting.
\end{remark}

The following theorem states accelerated learning properties of the proposed algorithm, and corresponds to the main result of this paper.
\begin{theorem}
  \label{theo:1}
  With $\lambda \geq 1$, $\kappa < 2\lambda$, $0 < \beta < 2$ and $\eta \geq \max\left\{\frac{\lambda(\kappa + 2\lambda) + \lambda\sqrt{5\kappa^2 - 4\lambda\kappa + 4\lambda^2}}{2\lambda - \kappa}, \frac{4\lambda(1 - \beta)^2}{\lambda - (1 - \beta)^2}\right\}$, the update law in \eqref{eq:2}, \eqref{eq:3} and \eqref{eq:4} will result in (i) $\vartheta_k - \theta^* \in \ell_\infty$, $\theta_k - \vartheta_k \in \ell_\infty$, and (ii) $\|\vartheta_k - \theta^*\|^2 + \|\theta_k - \vartheta_k\|^2 \leq \exp(-\mu k) V_0$, where $\mu$ is defined in \eqref{eq:mu}.
\end{theorem}
\begin{proof}
  Applying matrix inversion lemma to \eqref{eq:4}, we obtain
  \begin{equation}
    \label{eq:5}
    \begin{split}
      F_k^{-1} &= \frac{1}{\lambda}F_{k-1}^{-1} + \frac{\kappa}{\lambda}\frac{\phi_k\phi_k^\top}{\N_k - \kappa\phi_k^\top F_{k-1}\phi_k}
    \end{split}
  \end{equation}
  Consider the candidate Lyapunov function
  \begin{equation}
    \label{eq:lyap}
    V_k = \left(\vartheta_k - \theta^*\right)^\top F_{k-1}^{-1}\left(\vartheta_k - \theta^*\right) + (\theta_k - \vartheta_k)^\top F_{k-1}^{-1}(\theta_k - \vartheta_k)
  \end{equation}
  The increment $\Delta V_k := V_{k+1} - V_k$ may be expanded as
  \begin{align*}
    &\Delta V_k\\
    &= (\vartheta_{k+1} - \theta^*)^\top F_{k}^{-1}(\vartheta_{k+1} - \theta^*) + (\theta_{k+1} - \vartheta_{k+1})^\top F_{k}^{-1}(\theta_{k+1} - \vartheta_{k+1})\\
    &\quad -(\vartheta_k - \theta^*)^\top F_{k-1}^{-1}(\vartheta_k - \theta^*) - (\theta_k - \vartheta_k)^\top F_{k-1}^{-1}(\theta_k - \vartheta_k)\\
    &= \left[\vartheta_k - \theta^* - F_{k}\frac{\nabla L_k(\theta_{k+1})}{\N_k}\right]^\top F_{k}^{-1} \left[\vartheta_k - \theta^* - F_{k}\frac{\nabla L_k(\theta_{k+1})}{\N_k}\right]\\
    &\quad +\left[\theta_k - \beta(\theta_k - \vartheta_k) - \vartheta_k + F_{k}\frac{\nabla L_k(\theta_{k+1})}{\N_k}\right]^\top F_{k}^{-1} \left[\theta_k - \beta(\theta_k - \vartheta_k) - \vartheta_k + F_{k}\frac{\nabla L_k(\theta_{k+1})}{\N_k}\right]\\
    &\quad -(\vartheta_k - \theta^*)^\top F_{k-1}^{-1}(\vartheta_k - \theta^*) - (\theta_k - \vartheta_k)^\top F_{k-1}^{-1}(\theta_k - \vartheta_k)\\
    &= (\vartheta_k - \theta^*)^\top F_{k}^{-1}(\vartheta_k - \theta^*) + (1-\beta)^2(\theta_k - \vartheta_k)^\top F_{k}^{-1}(\theta_k - \vartheta_k) \\
    &\quad - \frac{2}{\N_k}(\vartheta_k - \theta^*)^\top\nabla L_k(\theta_{k+1}) + \frac{2(1 - \beta)}{\N_k}(\theta_k - \vartheta_k)^\top\nabla L_k(\theta_{k+1})\\
    &\quad + \frac{2}{\N_k^2}\left[\nabla L_k(\theta_{k+1})\right]^\top F_{k}\nabla L_k(\theta_{k+1})\\
    &\quad -(\vartheta_k - \theta^*)^\top F_{k-1}^{-1}(\vartheta_k - \theta^*) - (\theta_k - \vartheta_k)^\top F_{k-1}^{-1}(\theta_k - \vartheta_k)\\
    &= \frac{1}{\lambda}(\vartheta_k - \theta^*)^\top \!F_{k-1}^{-1}(\vartheta_k - \theta^*) + \frac{\kappa}{\lambda[1 + (\eta - \kappa)\phi_k^\top F_{k-1}\phi_k]}(\vartheta_k - \theta^*)^\top\phi_{k}\phi_{k}^\top(\vartheta_k - \theta^*) \\
    &\quad +\frac{(1-\beta)^2}{\lambda}(\theta_k - \vartheta_k)^\top F_{k-1}^{-1}(\theta_k - \vartheta_k) + \frac{\kappa(1-\beta)^2}{\lambda[1 + (\eta - \kappa)\phi_k^\top F_{k-1}\phi_k]} (\theta_k - \vartheta_k)^\top\phi_{k}\phi_{k}^\top(\theta_k - \vartheta_k)\\
    &\quad - \frac{2}{\N_k}(\vartheta_k - \theta^*)^\top\nabla L_k(\theta_{k+1}) + \frac{2(1 - \beta)}{\N_k}(\theta_k - \vartheta_k)^\top\nabla L_k(\theta_{k+1}) \\
    &\quad + \frac{2}{\N_k^2}\left[\nabla L_k(\theta_{k+1})\right]^\top F_{k}\nabla L_k(\theta_{k+1})\\
    &\quad -(\vartheta_k - \theta^*)^\top F_{k-1}^{-1}(\vartheta_k - \theta^*) - (\theta_k - \vartheta_k)^\top F_{k-1}^{-1}(\theta_k - \vartheta_k)
  \end{align*}
  Now substitute $\nabla L_k(\theta_{k+1}) = \phi_k\phi_k^\top\tilde{\theta}_{k+1}$ into the above, we get
  \begin{align*}
    \Delta V_k
    &= \frac{1}{\lambda}(\vartheta_k - \theta^*)^\top F_{k-1}^{-1}(\vartheta_k - \theta^*) + \frac{\kappa}{\lambda[1 + (\eta - \kappa)\phi_k^\top F_{k-1}\phi_k]}\|(\vartheta_k - \theta^*)^\top\phi_{k}\|^2\\
    &\quad +\frac{(1-\beta)^2}{\lambda}(\theta_k - \vartheta_k)^\top F_{k-1}^{-1}(\theta_k - \vartheta_k) + \frac{\kappa(1-\beta)^2}{\lambda[1 + (\eta - \kappa)\phi_k^\top F_{k-1}\phi_k]} \|(\theta_k - \vartheta_k)^\top\phi_{k}\|^2\\
    &\quad -\frac{2}{\N_k}(\vartheta_k - \theta^*)^\top\phi_k\phi_k^\top\tilde{\theta}_{k+1} + \frac{2(1 - \beta)}{\N_k}(\theta_k - \vartheta_k)^\top\phi_k\phi_k^\top\tilde{\theta}_{k+1} \\
    &\quad + \frac{2}{\N_k^2}\left[\nabla L_k(\theta_{k+1})\right]^\top F_{k}\nabla L_k(\theta_{k+1})\\
    &\quad -(\vartheta_k - \theta^*)^\top F_{k-1}^{-1}(\vartheta_k - \theta^*) - (\theta_k - \vartheta_k)^\top F_{k-1}^{-1}(\theta_k - \vartheta_k)
  \end{align*}
  Since $\tilde{\theta}_{k+1} = \theta_{k+1} - \vartheta_k + \vartheta_k - \theta^* = (1 - \beta)(\theta_k - \vartheta_k) + (\vartheta_k - \theta^*)$,
  \begin{align*}
    \Delta V_k
    &= \frac{1}{\lambda}(\vartheta_k - \theta^*)^\top F_{k-1}^{-1}(\vartheta_k - \theta^*) + \frac{\kappa}{\lambda[1 + (\eta - \kappa)\phi_k^\top F_{k-1}\phi_k]}\|(\vartheta_k - \theta^*)^\top\phi_{k}\|^2\\
    &\quad +\frac{(1-\beta)^2}{\lambda}(\theta_k - \vartheta_k)^\top F_{k-1}^{-1}(\theta_k - \vartheta_k) + \frac{\kappa(1-\beta)^2}{\lambda[1 + (\eta - \kappa)\phi_k^\top F_{k-1}\phi_k]} \|(\theta_k - \vartheta_k)^\top\phi_{k}\|^2\\
    &\quad -\frac{2}{\N_k}\|(\vartheta_k - \theta^*)^\top\phi_k\|^2 - \frac{2(1 - \beta)}{\N_k}(\vartheta_k - \theta^*)^\top\phi_k\phi_k^\top(\theta_k - \vartheta_k)\\
    &\quad + \frac{2(1 - \beta)^2}{\N_k}\|(\theta_k - \vartheta_k)^\top\phi_k\|^2 + \frac{2(1 - \beta)}{\N_k}(\theta_k - \vartheta_k)^\top\phi_k\phi_k^\top(\vartheta_k - \theta^*)\\
    &\quad + \frac{2}{\N_k^2}\left[\nabla L_k(\theta_{k+1})\right]^\top F_{k}\nabla L_k(\theta_{k+1})\\
    &\quad -(\vartheta_k - \theta^*)^\top F_{k-1}^{-1}(\vartheta_k - \theta^*) - (\theta_k - \vartheta_k)^\top F_{k-1}^{-1}(\theta_k - \vartheta_k)\\
    &= \frac{1}{\lambda}(\vartheta_k - \theta^*)^\top F_{k-1}^{-1}(\vartheta_k - \theta^*) + \frac{\kappa}{\lambda[1 + (\eta - \kappa)\phi_k^\top F_{k-1}\phi_k]}\|(\vartheta_k - \theta^*)^\top\phi_{k}\|^2\\
    &\quad +\frac{(1-\beta)^2}{\lambda}(\theta_k - \vartheta_k)^\top F_{k-1}^{-1}(\theta_k - \vartheta_k) + \frac{\kappa(1-\beta)^2}{\lambda[1 + (\eta - \kappa)\phi_k^\top F_{k-1}\phi_k]} \|(\theta_k - \vartheta_k)^\top\phi_{k}\|^2\\
    &\quad - \frac{2}{\N_k}\|(\vartheta_k - \theta^*)^\top\phi_k\|^2 + \frac{2(1 - \beta)^2}{\N_k}\|(\theta_k - \vartheta_k)^\top\phi_k\|^2 + \frac{2}{\N_k^2}\tilde{\theta}_{k+1}^\top\phi_k\phi_k^\top F_k\phi_k\phi_k^\top\tilde{\theta}_{k+1}\\
    &\quad -(\vartheta_k - \theta^*)^\top F_{k-1}^{-1}(\vartheta_k - \theta^*) - (\theta_k - \vartheta_k)^\top F_{k-1}^{-1}(\theta_k - \vartheta_k)
  \end{align*}

  Let $A_k = \frac{2\lambda}{\eta^2\N_k^3}[(\eta - \kappa)\N_k + \kappa](\N_k - 1)$, the above becomes
  \begin{align*}
    \Delta V_k
    &= \frac{1}{\lambda}(\vartheta_k - \theta^*)^\top F_{k-1}^{-1}(\vartheta_k - \theta^*) + \frac{\kappa}{\lambda[1 + (\eta - \kappa)\phi_k^\top F_{k-1}\phi_k]}\|(\vartheta_k - \theta^*)^\top\phi_{k}\|^2\\
    &\quad +\frac{(1-\beta)^2}{\lambda}(\theta_k - \vartheta_k)^\top F_{k-1}^{-1}(\theta_k - \vartheta_k) + \frac{\kappa(1-\beta)^2}{\lambda[1 + (\eta - \kappa)\phi_k^\top F_{k-1}\phi_k]} \|(\theta_k - \vartheta_k)^\top\phi_{k}\|^2\\
    &\quad - \frac{2}{\N_k}\|(\vartheta_k - \theta^*)^\top\phi_k\|^2 + \frac{2(1 - \beta)^2}{\N_k}\|(\theta_k - \vartheta_k)^\top\phi_k\|^2\\
    &\quad + A_k\left[\|(\theta_k - \vartheta_k)^\top\phi_k\|^2 + \|(\vartheta_k - \theta^*)\phi_k\|^2 \right. + \left. 2(\theta_k - \vartheta_k)^\top\phi_k\phi_k^\top(\vartheta_k - \theta^*)\right]\\
    &\quad -(\vartheta_k - \theta^*)^\top F_{k-1}^{-1}(\vartheta_k - \theta^*) - (\theta_k - \vartheta_k)^\top F_{k-1}^{-1}(\theta_k - \vartheta_k)
  \end{align*}

  Combining similar terms,
  \begin{align*}
    \Delta V_k
    &= \left(\frac{1}{\lambda} - 1\right)(\vartheta_k - \theta^*)^\top F_{k-1}^{-1}(\vartheta_k - \theta^*) + \left[\frac{(1-\beta)^2}{\lambda} - 1\right](\theta_k - \vartheta_k)^\top F_{k-1}^{-1}(\theta_k - \vartheta_k)\\
    &\quad +\left\{\frac{\kappa\eta}{\lambda[(\eta - \kappa)\N_k + \kappa]} - \frac{2}{\N_k} + A_k\right\}\|(\vartheta_k - \theta^*)^\top\phi_k\|^2\\
    &\quad +\left\{\frac{\kappa\eta}{\lambda[(\eta - \kappa)\N_k + \kappa]} + \frac{2}{\N_k} + A_k\right\} (1 - \beta)^2\|(\theta_k - \vartheta_k)^\top\phi_k\|^2\\
    &\quad +2A_k(1 - \beta)(\theta_k - \vartheta_k)^\top\phi_k\phi_k^\top(\vartheta_k - \theta^*)\\
    &= \left(\frac{1}{\lambda} - 1\right)(\vartheta_k - \theta^*)^\top F_{k-1}^{-1}(\vartheta_k - \theta^*) + \left[\frac{(1-\beta)^2}{\lambda} - 1\right](\theta_k - \vartheta_k)^\top F_{k-1}^{-1}(\theta_k - \vartheta_k)\\
    &\quad +\left\{\frac{\kappa\eta}{\lambda[(\eta - \kappa)\N_k + \kappa]} - \frac{2}{\N_k} + 2A_k\right\}\|(\vartheta_k - \theta^*)^\top\phi_k\|^2\\
    &\quad +\left\{\frac{\kappa\eta}{\lambda[(\eta - \kappa)\N_k + \kappa]} + \frac{2}{\N_k} + 2A_k\right\} (1 - \beta)^2\|(\theta_k - \vartheta_k)^\top\phi_k\|^2\\
    &\quad -A_k[(\vartheta_k - \theta^*)^\top\phi_k - (1 - \beta)(\theta_k - \vartheta_k)^\top\phi_k]^2
  \end{align*}

  From Cauchy-Schwarz inequality,
  \begin{align*}
    \frac{1}{\N_k}\|(\theta_k - \vartheta_k)^\top\phi_k\|^2 \leq \frac{1}{\eta}(\theta_k - \vartheta_k)^\top F_{k-1}^{-1}(\theta_k - \vartheta_k) 
  \end{align*}
  Therefore
  \begin{align*}
    \Delta V_k
    &\leq \left(\frac{1}{\lambda} - 1\right)(\vartheta_k - \theta^*)^\top F_{k-1}^{-1}(\vartheta_k - \theta^*) + \left[\frac{(1-\beta)^2}{\lambda} - 1\right](\theta_k - \vartheta_k)^\top F_{k-1}^{-1}(\theta_k - \vartheta_k)\\
    &\quad +\left\{\frac{\kappa\eta}{\lambda[(\eta - \kappa)\N_k + \kappa]} - \frac{2}{\N_k} + 2A_k\right\}\|(\vartheta_k - \theta^*)^\top\phi_k\|^2\\
    &\quad +\left\{\frac{\kappa\eta\N_k}{\lambda[(\eta - \kappa)\N_k + \kappa]} + 2 + 2A_k\N_k\right\} \frac{(1 - \beta)^2}{\eta} (\theta_k - \vartheta_k)^\top F_{k-1}^{-1}(\theta_k - \vartheta_k)\\
    &\quad -A_k[(\vartheta_k - \theta^*)^\top\phi_k - (1 - \beta)(\theta_k - \vartheta_k)^\top\phi_k]^2
  \end{align*}

  Since
  \begin{align*}
    \frac{\kappa\eta\N_k}{\lambda[(\eta - \kappa)\N_k + \kappa]} \leq \frac{\kappa\eta}{\lambda(\eta - \kappa)},
  \end{align*}
  and
  \begin{align*}
    A_k\N_k \leq \frac{2\lambda}{\eta}
  \end{align*}
  and also since $\lambda \geq 1$, $\kappa < 2\lambda$, $0 < \beta < 2$, $\eta \geq \max\left\{\frac{\lambda(\kappa + 2\lambda) + \lambda\sqrt{5\kappa^2 - 4\lambda\kappa + 4\lambda^2}}{2\lambda - \kappa}, \frac{4\lambda(1 - \beta)^2}{\lambda - (1 - \beta)^2}\right\}$, under persistent excitation and from Lemma \ref{lemma:1}, we get
  \begin{align*}
    \Delta V_k
    &\leq -c_1\|\vartheta_k - \theta^*\|^2 - c_2\|\theta_k - \vartheta_k\|^2\\
    &\quad -A_k[(\vartheta_k - \theta^*)^\top\phi_k - (1 - \beta)(\theta_k - \vartheta_k)^\top\phi_k]^2\\
    &\leq 0
  \end{align*}
  where $c_1$ and $c_2$ are defined in \eqref{eq:c1}-\eqref{eq:c2}. Thus $\vartheta_k - \theta^*\in\ell_\infty$ and $\theta_k - \vartheta_k\in\ell_\infty$. Furthermore,
  \begin{align*}
    \Delta V_k &\leq -\mu \left(\|\vartheta_k - \theta^*\|^2 + \|\theta_k - \vartheta_k\|^2\right),
  \end{align*}
  where $\mu$ is defined in \eqref{eq:mu}.

  Since $F_{\max}^{-1}(\|\vartheta_k - \theta^*\|^2 + \|\theta_k - \vartheta_k\|^2) \leq V_k \leq F_{\min}^{-1}(\|\vartheta_k - \theta^*\|^2 + \|\theta_k - \vartheta_k\|^2)$, and $\Delta V_k \leq -\mu\left(\|\vartheta_k - \theta^*\|^2 + \|\theta_k - \vartheta_k\|^2\right)$, according to Theorem \ref{theo:geo-stable}, $\|\vartheta_k - \theta^*\|\rightarrow 0$ and $\|\theta_k - \vartheta_k\|\rightarrow 0$ globally uniformly exponentially fast.
\end{proof}

\begin{remark}
  Theorem \ref{theo:1} states that under PE, both $\|\vartheta_k - \theta^*\|$ and $\|\theta_k - \vartheta_k\|$ go to zero exponentially fast. Together, $\|\vartheta_k - \theta^*\|^2 + \|\theta_k - \vartheta_k\|^2$ goes to zero exponentially fast.
\end{remark}

\begin{remark}
  Note that directly applying the covariance matrix update in RLS with forgetting to \eqref{eq:2} leads to Lyapunov stability in only very limited cases such as when $F_0$ is small. It is the introduction of new hyperparameters in the update of $F_k$ that gives more flexibility in the choice of hyperparameters and made global uniform exponential convergence possible.
\end{remark}

\begin{remark}
  From \eqref{eq:4} and \eqref{eq:5}, under weak or no excitation, the eigenvalues of $F_k$ keep increasing. To avoid the values of $F_k$ getting too large when excitation is weak, a barrier function such as the one in \cite{cui2022} can be applied to the update of $F_k$ in \eqref{eq:4}, or a variable forgetting factor such as the one introduced in \cite{bruce20a} can be considered. A complete proof of this case is beyond the scope of this paper and will be addressed in future work.
\end{remark}

\begin{remark}
  Exponential decrease of $V_k$ still happens when there is no persistent excitation in the regressors. However, that does not mean $\theta_k$ keeps converging to its true value. As an extreme example, when $\phi_k = 0$, $V_k$ converges to zero exponentially, due to the exponential increase of $F_k$, but both $\theta_k$ and $\vartheta_k$ are not changing.
\end{remark}


\begin{remark}
	As we will show in Section \ref{sec:numer-sim}, the additional benefits of the proposed algorithm compared to RLS with forgetting become apparent under weak excitation signals. The presence of momentum helps boost both output error and parameter error convergence towards zero.
\end{remark}

\section{Numerical Simulations}
\label{sec:numer-sim}
A linear discrete system is given by
\begin{equation}
	\label{eq:21}
	G(\bm{q}) = \frac{-0.6213\bm{q} + 0.5839}{\bm{q}^2 - 1.8403\bm{q} + 0.8591}.
\end{equation}
The problem is that the coefficients of \eqref{eq:21} are unknown and we use the proposed algorithm to identify them. 

\subsection{Exponential Parameter Convergence Under Persistent Excitation}
For identification, we apply the following signal as an input
\begin{equation}
  \label{eq:22}
  u(k) = 1 + \sin\left(\frac{3\pi k}{4}\right) + \sin\left(\frac{2\pi k}{5}\right) + \sin\left(\frac{\pi k}{5}\right).
\end{equation}

The proposed algorithm in \eqref{eq:1}-\eqref{eq:3} is tested. The forgetting factor in our algorithm is set to be $\lambda = 1.01$ and the initial values of the learning rate matrices are set to be $100 I$. The regressors and estimated parameters are all set to be zero at the initial step. For the proposed algorithm, we set $\beta = 0.6$, $\kappa = 0.7$ and $\eta = 3.6$, satisfying the assumptions in Theorem \ref{theo:1}.

Figure \ref{fig:1} and Figure \ref{fig:2} show output errors and parameter errors in semi-log scale, respectively. Figure \ref{fig:3} shows histories of the learning rate matrix in the algorithm. In this case, both the output error and parameter error converge exponentially towards zero.

\begin{figure}[!htb]
  \centering
  \includegraphics[width=0.7\linewidth]{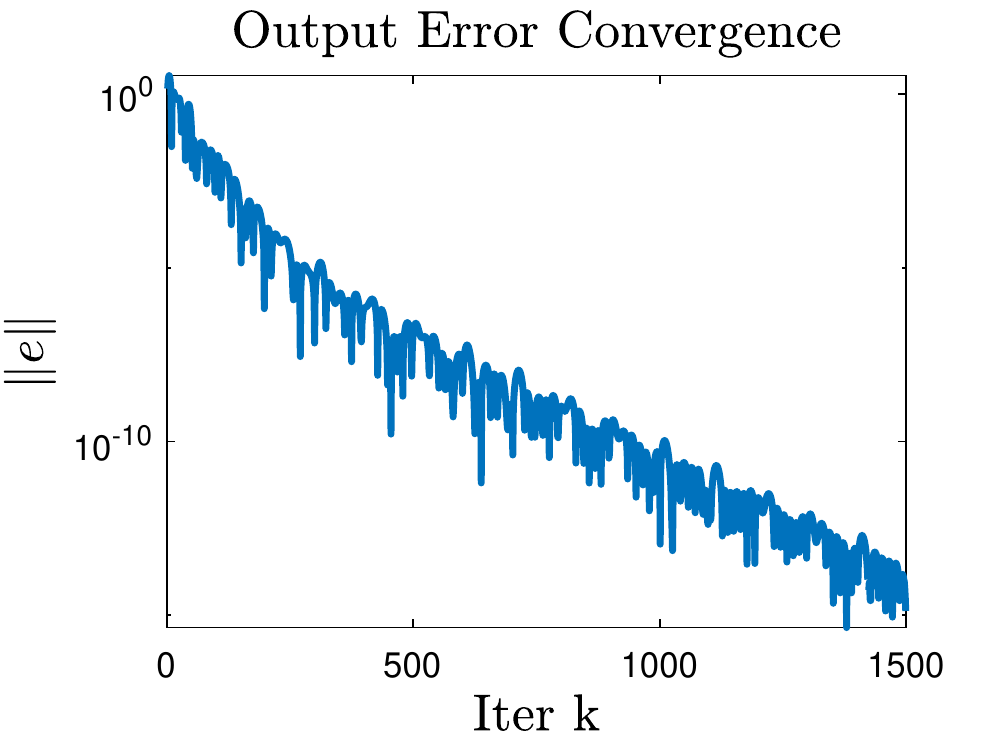}
  \caption{Output error $\|e\|$ of the proposed algorithm under PE.}
  \label{fig:1}
\end{figure}

\begin{figure}[!htb]
  \centering
  \includegraphics[width=0.7\linewidth]{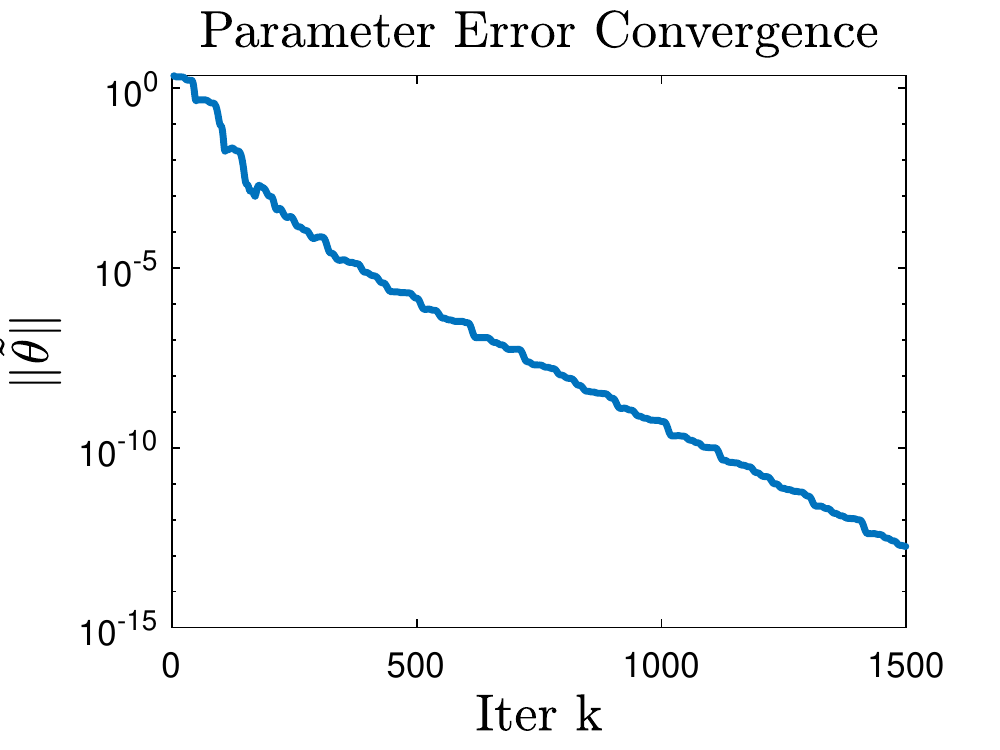}
  \caption{Parameter error $\|\tilde{\theta}\|$ of the proposed algorithm under PE.}
  \label{fig:2}
\end{figure}

\begin{figure}[!htb]
  \centering
  \includegraphics[width=0.7\linewidth]{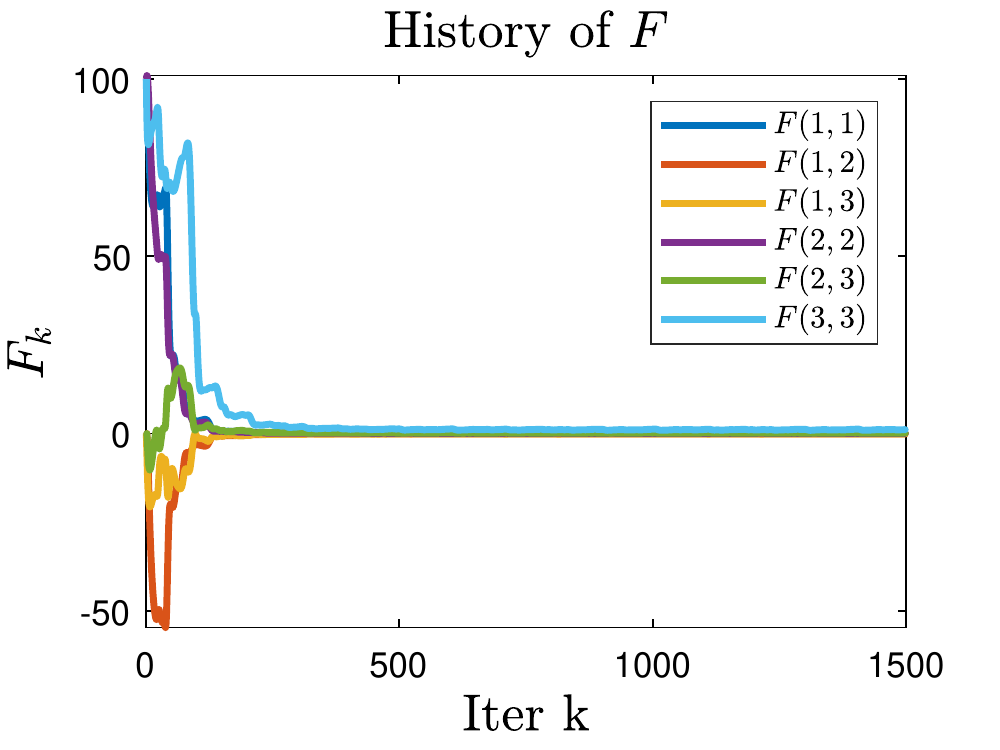}
  \caption{Learning rate matrix histories of the proposed algorithm.}
  \label{fig:3}
\end{figure}

\subsection{Comparison To RLS With Forgetting Under Increasingly Weaker Excitation}
RLS algorithm with a forgetting factor (RLS-FF) has been a widely used algorithm for online parameter estimation, and is quite similar to the algorithm proposed in this paper. In RLS-FF, the parameters are updated as follows \cite{Goodwin_1984}:
\begin{align}
	\label{eq:rls_p}
	P_{k} &= \frac{1}{\bar\lambda} P_{k-1} - \frac{P_{k-1}\phi_k\phi_k^\top P_{k-1}}{\bar\lambda + \phi_k^\top P_{k-1} \phi_k}, \\
	\label{eq:rls_v}
	\theta_{k+1} &= \theta_k + \frac{P_{k-1}\phi_k(y_{k+1} - \phi_k^\top \theta_k)}{\bar\lambda + \phi_k^\top P_{k-1}\phi_k},
\end{align}
where $P_k$ is the covariance matrix and $\bar\lambda$ is the forgetting factor. The difference between $P_k$ in \eqref{eq:rls_p} and $F_k$ in \eqref{eq:4} can be seen to be slight, but still makes a distinction as shown below.  The other major difference is an HT aspect (see \eqref{eq:2}-\eqref{eq:3}) in our algorithm, while a gradient descent idea is employed in \eqref{eq:rls_v}.

To demonstrate the added benefits of our algorithm compared to RLS with forgetting, we apply the following signal as an input to the system in \eqref{eq:21}:
\begin{align*}
	&u(k) = \\
	&1 + \exp(-0.03k)\!\left[ \sin\left(\frac{3\pi k}{4}\right) + \sin\left(\frac{2\pi k}{5}\right) + \sin\left(\frac{\pi k}{5}\right)\right]
\end{align*}
which is an increasingly weaker excitation signal. To ensure a fair comparison, we choose hyperparameters and initial values to be $\bar\lambda = 1 / \lambda = 0.99$, $\kappa = 1.06$, $\eta=3$, $\beta=0.5$, $F_0 = P_0 = 100 I$ and $\theta_0 = 0$ such that the time-varying matrices in the two algorithms are roughly the same as $k$ increases, see Figure \ref{fig:4}. Figure \ref{fig:5} shows the output error comparison and Figure \ref{fig:6} shows the parameter error comparison. Our algorithm demonstrates faster convergence results in this case. We speculate that the main reason for this faster convergence is the momentum term in the HB method, which in turn allows fast decrease in output error $e$. The time-varying $F_k$ exploits persistent excitation and ensures that this fast decrease in performance error translates into fast decrease in learning error.

\begin{figure}
  \centering
  \includegraphics[width=0.7\linewidth]{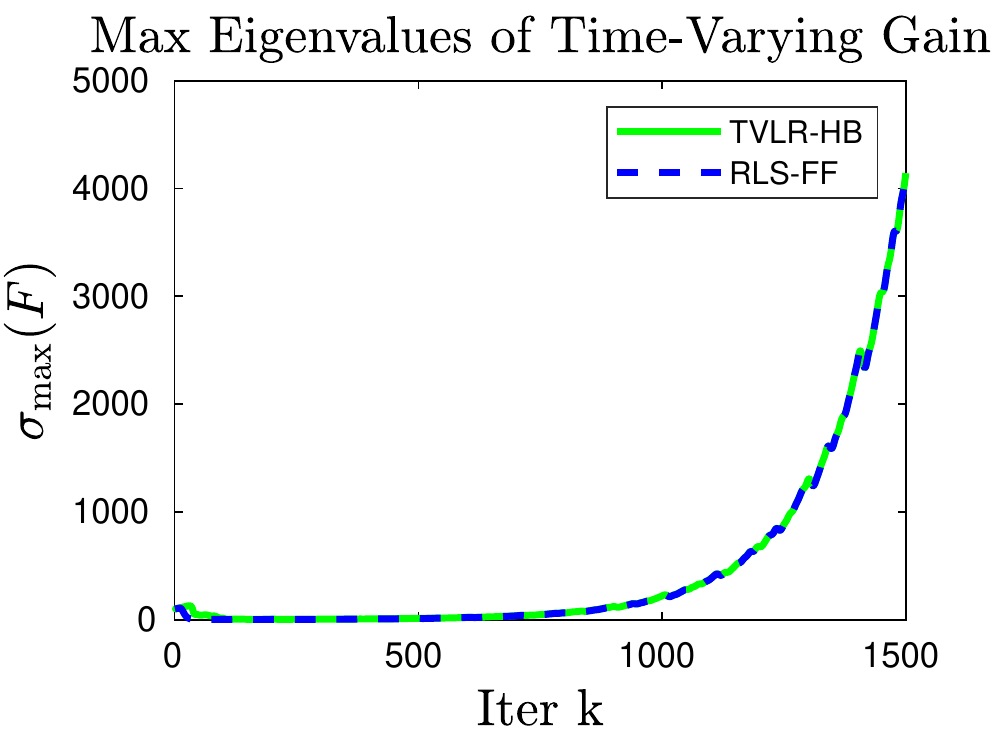}
  \caption{Comparison of the max eigenvalues of time-varying gain matrices between the proposed algorithm and RLS with forgetting.}
  \label{fig:4}
\end{figure}

\begin{figure}
  \centering
  \includegraphics[width=0.7\linewidth]{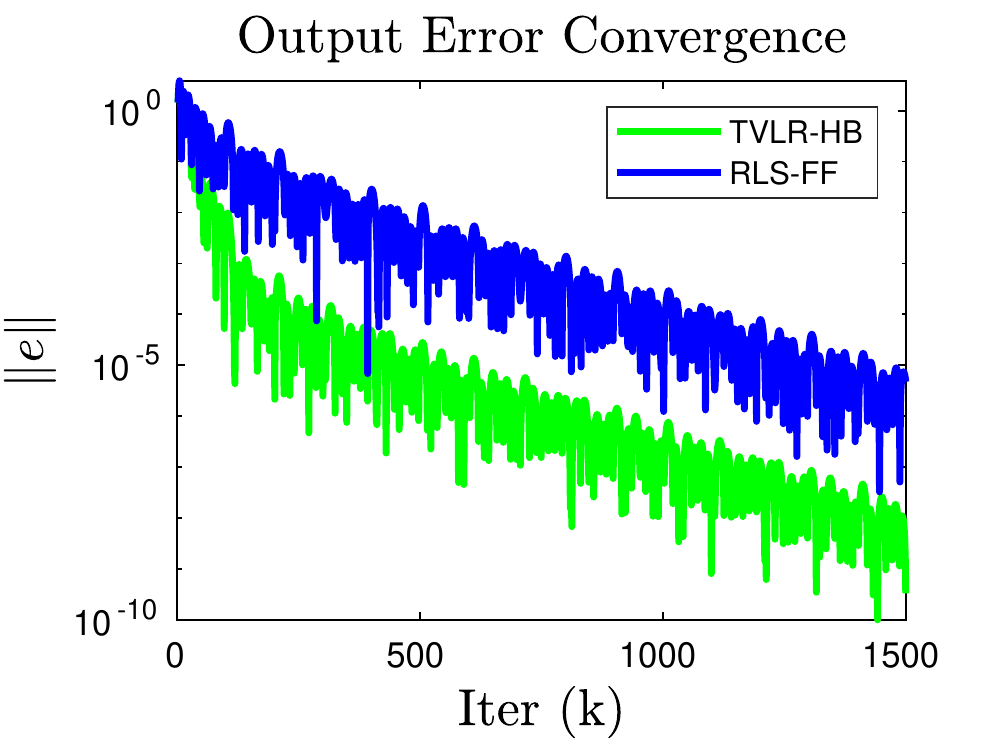}
  \caption{Comparison of output error convergence between the proposed algorithm and RLS with forgetting.}
  \label{fig:5}
\end{figure}

\begin{figure}
  \centering
  \includegraphics[width=0.7\linewidth]{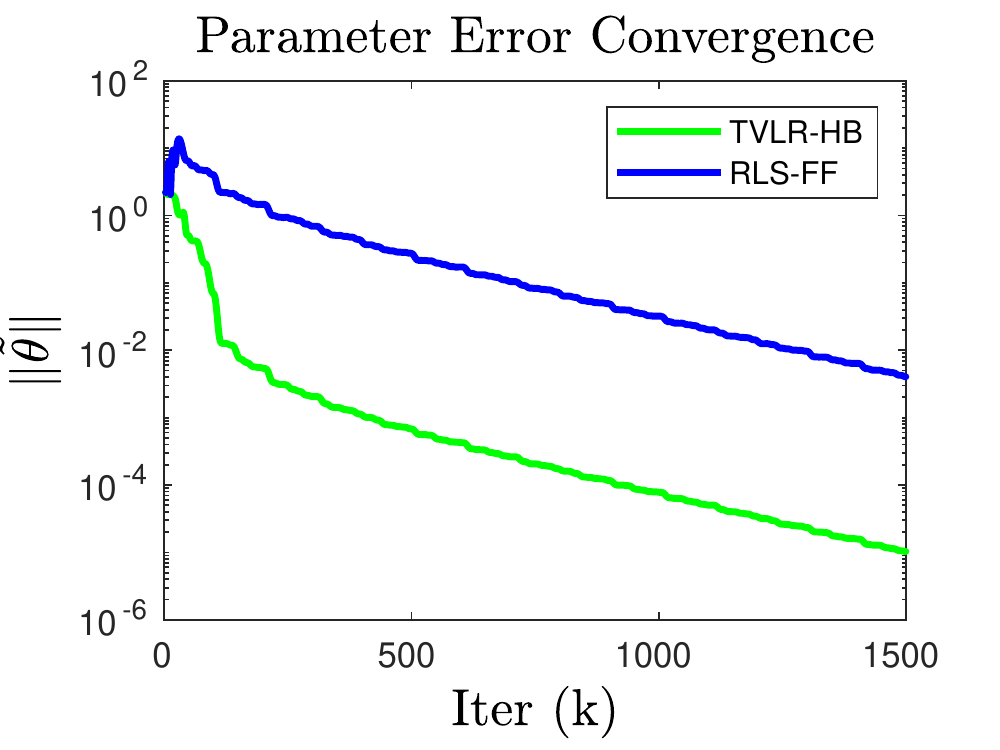}
  \caption{Comparison of parameter error convergence between the proposed algorithm and RLS with forgetting.}
  \label{fig:6}
\end{figure}

\section{Conclusion}
\label{sec:conclusion}
We introduced an online parameter estimation algorithm that adopts the ideas of momentum and time-varying learning rate. Under persistent excitation, the algorithm results in exponential convergence of the parameter error towards zero. Compared to recursive least squares with a forgetting factor, the presence of momentum in the update provides more flexibility. As shown in the simulation results, this flexibility translates into better performance and learning when the excitation is weak. Similar to recursive least squares with forgetting, one disadvantage of the algorithm is the unboundedness of the learning rate matrix when persistent excitation is not assured. In that case, projection operators need to be included to regulate the behavior of the learning rate matrix, which will be considered in future works.

\bibliographystyle{IEEEtran}
\bibliography{IEEEabrv,References}
\end{document}